\documentclass[a4paper,12pt]{article}

\usepackage{amssymb,amsmath,amsthm,mathrsfs,latexsym,times}

\usepackage[authoryear,round]{natbib}

\usepackage[pdftex]{hyperref}
\hypersetup{plainpages=True, pdfstartview=FitV,
	colorlinks=true,linkcolor=blue,citecolor=blue}

\newtheorem{thm}{Theorem}[section]
\newtheorem{cor}[thm]{Corollary}
\newtheorem{lem}[thm]{Lemma}

\theoremstyle{remark}
\newtheorem{rem}[thm]{Remark}
\def\Po{{\mathscr{P}}}

\def\dtv{d_{T\!V}}

	\title{Poisson approximation in \(\chi^2\) distance by Chen-Stein approach}
\author{Vytas Zacharovas\\
	Institute of Computer Science\\
	Vilnius University\\
	Naugarduko 24, Vilnius, Lithuania\\
	E-mail: vytas.zacharovas@mif.vu.lt}

\begin{document}
\maketitle
\begin{abstract}
	\noindent 
	The main purpose of the paper is to investigate the possibility of  applying Chen-Stein approach to estimate the \(\chi^2\) distance between Poisson distribution and a sum of independent indicators. Earlier results concerning \(\chi^2\) distance between above mentioned distributions either used analytical approach heavily based on the analysis of the generating functions or on rather lengthy and complicated elementary calculations. Applying Chen-Stein approach we succeed in providing a very quick proof of upper bounds for \(\chi^2\) distance that are of comparable strength to the  earlier estimates obtained by other approaches.
	
\end{abstract}
\section{Introduction}
In what follows we will denote as \(I_1,I_2,\ldots,I_n\) a series of independent random indicators where indicator \(I_j\) takes values \(1\) and \(0\)  with probabilites \(p_j\) and \(1-p_j\) correspondingly. We will also dentote as \(S_n\) the sum
\[
S_n=I_1+I_2+\cdots+I_n
\]
of random independent indicators \(I_n\). The expectation \(\lambda\) of this random variable can be easily computed
\[
\lambda=p_1+p_2+\cdots+p_n.
\]
Later we will also use notation \(\lambda_2\) for the sum of squares
\[
\lambda_2=p_1^2+p_2^2+\cdots+p_n^2
\]
and cubes
\[
\lambda_3=p_1^3+p_2^3+\cdots+p_n^3.
\]
A lot of research has been done evaluating the total variation distance defined as
\[
\dtv(\mathscr{L}(S_n),\Po(\lambda)) := \frac12
\sum_{j\ge0} \left|\mathbb{P}(S_n=j) - e^{-\lambda}
\frac{\lambda^j}{j!}\right|
\]
between the distribution \(\mathscr{L}(S_n)\) of \(S_n\) and Poisson distribution \(\Po(\lambda)\) with parameter \(\lambda\) equal to the mean value of \(S_n\). One of the best known results in this direction is the Barbour-Hall inequality 
\begin{equation}
	\label{barbour_hall_inequality}
\dtv(\mathscr{L}(S_n),\Po(\lambda))\leqslant (1-e^{-\lambda}) \frac{\lambda_2}{\lambda}
\end{equation}
proved in \cite{barbour_hall_1984}. We refer reader interested in a more complete account of history of this direction of research to the book \cite{barbour_holst_janson_1992} or our paper \cite{zacharovas_hwang_2010}.

The research into properties of the \(\chi^2\) distance between \(\mathscr{L}(S_n)\) and \(\Po(\lambda)\) defined as
\[
d_{\chi^2}(\mathscr{L}(S_n),\Po(\lambda))
:=\sum_{m\ge0}\left|\frac{\mathbb{P}(S_n=m)}
{e^{-\lambda}\frac{\lambda^m}{m!}}-1\right|^2
e^{-\lambda}\frac{\lambda^m}{m!}
\]
has a more recent history starting with \cite{borisov_vorozheikin_2008} who using  an elementary probabilistic approach proved  that
\[
\frac{d_{\chi^2}(\mathscr{L}(S_n),\Po(\lambda))}{\left(\frac{\lambda_2}{\lambda}\right)^2}\to \frac{1}{2}
\]
if \(\lambda^6\lambda_2\to 0\). This estimate was later improved by \cite{zacharovas_hwang_2010} who showed that 
\[
d_{\chi^2}(\mathscr{L}(S_n),\Po(\lambda))
=\left(\frac{1}{\sqrt{1-\left(\frac{\lambda_2}{\lambda}\right)^2}}-1\right)
\left(1+O\left(\frac{\lambda_3}{\lambda_2
	\sqrt{\lambda}\left(1-\frac{\lambda_2}{\lambda}\right)^5}\right)\right).
\]
whenever \(\frac{\lambda_2}{\lambda}<1\) and the constant in the symbol \(O(\ldots)\) is absolute and can be made explicit if needed.  This estimate as a special case implies that if \(\lambda\to \infty\) and \(\frac{\lambda_2}{\lambda}\leqslant c<1\) for some fixed constant \(c\)  then
\[
d_{\chi^2}(\mathscr{L}(S_n),\Po(\lambda))
=\left(\frac{1}{\sqrt{1-\left(\frac{\lambda_2}{\lambda}\right)^2}}-1\right)\bigl(1+o(1)\bigr)
\]
A series of upper bounds for \(\chi^2\) distance has been obtained in the same paper, as an example of which we present the following inequality
\begin{equation}
	\label{z_hw_ineq}
d_{\chi^2}(\mathscr{L}(S_n),\Po(\lambda))
\leqslant 2(\sqrt{e}-1)^2\frac{\left(\frac{\lambda_2}{\lambda}\right)^2}{\left(1-\frac{\lambda_2}{\lambda}\right)^{3}}.
\end{equation}

The main tool of work of the above mentioned paper was the following integral expression  for the \(\chi^2\) distance
\begin{equation}
	\label{Charlier-Parseval_integral_form}
	\begin{split}
&d_{\chi^2}(\mathscr{L}(S_n),\Po(\lambda))
\\
&\quad=\frac{ 1}{2\pi}\int_{0}^{\infty}\int_{-\pi}^\pi
\left|\prod_{1\le j\le n}(1+p_j\sqrt{r/\lambda}e^{it})
e^{-p_j\sqrt{r/\lambda}e^{it}}-1\right|^2\,dt\,e^{-r}\,dr
	\end{split}
\end{equation}
followed by the appropriate  analysis of the double integral in the above expression.

Recently \cite{bobkov_chistyakov_goetze_2019_1} and \cite{bobkov_chistyakov_goetze_2019_2} analysed asymptotic behavior of \(\chi^2\) distance in a more straightforward manner by obtaining first sufficiently good local estimates of probabilities \(\mathbb{P}(S_n=m)\) for \(m\) belonging to various regions and then summing them up in the sum in the definition of \(\chi^2\) distance. The main tool used for obtaining such estimates was the usual expression of \(\mathbb{P}(S_n=m)\) as a classical integral of characteristic function of \(S_n\) and then applying saddle point method for the analysis of the resulting integral. They showed that there exist such constants \(C_1\) and \(C_2\) that
\[
C_1 \left(\frac{\lambda_2}{\lambda}\right)^2\sqrt{\frac{\max\{1,\lambda\}}{\max\{1,\lambda-\lambda_2\}}}\leqslant d_{\chi^2}(\mathscr{L}(S_n),\Po(\lambda))\leqslant C_2 \left(\frac{\lambda_2}{\lambda}\right)^2\sqrt{\frac{\max\{1,\lambda\}}{\max\{1,\lambda-\lambda_2\}}}
\]
Unfortunately due to elementary nature of their approach consisting of multiple steps of evaluations that make keeping track of the constants involved very difficult, their estimates of constants involved are very crude. For example they show that the above inequality holds if \(C_0=10^{-8}\) and \(C_2=5.6\cdot 10^7\) 
Applying their approach the authors provided another proof of the  inequality (\ref{z_hw_ineq}) with constant \(2(\sqrt{e}-1)^2\)   replaced by a significantly larger value \(7\cdot 10^6\). 

All of the  described approaches have a significant drawback of being difficult or impossible to apply to sums of indicators consisting of dependent variables. Stein's approach has a big advantage of being able to be relatively easily extendable to the sums of dependent indicators. The main idea of this approach when applied to the analysis of total variation distance is as follows. First one can show that the total variation distance between \(S_n\) and \(\Po(\lambda)\) can be represented as
\begin{equation}
	\label{Chen-Stein_identity}
\dtv(\mathscr{L}(S_n),\Po(\lambda))=\mathbb{E}\bigl(\lambda g(S_n+1)-S_ng(S_n)\bigr)
\end{equation}
where \(g(x)\) is a solution of a certain recurrence relation that is shown to satisfy the property
\begin{equation}
	\label{Chen_Stein_function_diff_ineq}
|g(x+1)-g(x)|\leqslant \frac{1-e^{-\lambda}}{\lambda}
\end{equation}
for all \(x\).
Since it can be easily shown that the right side of  (\ref{Chen-Stein_identity}) becomes zero if we replace \(S_n\) by Poisson random variable  \(W\) with mean \(\lambda\) for any function \(f(x)\) in place of \(g(x)\)
\[
\mathbb{E}\bigl(\lambda f(W+1)-Wf(W)\bigr)=0
\]
so intuition is that if \(S_n\) is close to \(W\) in distribution then the right side of  (\ref{Chen-Stein_identity}) will be close to zero.

 Further in the case of independent random indicators the right hand side of the Stein-Chen expression for the total variation distance can be simplified to
\[
\dtv(\mathscr{L}(S_n),\Po(\lambda))=\sum_{j=1}^{n}p_j^2\mathbb{E}\bigl(g(S_{n,j}+2)-g(S_{n,j}+1)\bigr)
\]
where \(S_{n,j}=S_n-I_j\). Where application of the estimate for the inequality for \(|g(x+1)-g(x)|\) immediately leads to Barbour-Hall inequality (\ref{barbour_hall_inequality}).

Our application of the Stein-Chen metod idea starts with Charlier-Parseval identity 
\begin{equation}
	\label{Charlier_Parseval_for_chi}
d_{\chi^2}(\mathscr{L}(S_n),\Po(\lambda))=
\sum_{k\ge1}\frac{\lambda^k}{k!}
\bigl|\mathbb{E}C_k(\lambda,S_n)\bigr|^2
\end{equation}
where \(C_k(\lambda,m)\) are Charlier polynomials of variable \(m\). These polynomials can be shown to satisfy the properties
\begin{equation}
	\label{recurrence_for_charlier}
	C_{k+1}(\lambda,m)=\frac{m}{\lambda}C_k(\lambda,m-1)
	-C_k(\lambda,m),
\end{equation}
\begin{equation}
	\label{charlier_difference}
	C_k(\lambda,m+1)-C_k(\lambda,m)=\frac{k}{\lambda}C_{k-1}(\lambda,m)
\end{equation}
which together with initial condition
\[
C_0(\lambda,m)\equiv 1
\] can be used for recurrent calculation of Charlier polynomials. For example setting in (\ref{recurrence_for_charlier}) parameter \(k\) equal to zero we obtain
\[
C_1(\lambda,m)=\frac{m}{\lambda}-1.
\]
The first of these properties when applied to the the expectations present in the Charlier-Parseval  identity yield the expression
\[
\mathbb{E}C_k(\lambda,S_n)=-\frac{1}{\lambda}\mathbb{E}\bigl(
\lambda C_{k-1}(\lambda,S_n)-S_nC_{k-1}(\lambda,S_n-1)\bigr).
\]
The right side of the above identity is exactly the right hand side of Chen-Stein identity (\ref{Chen-Stein_identity}) multiplied by factor \(-1/\lambda\) with \(C_{k-1}(\lambda,x-1)\) standing
in place of \(g(x)\). The second identity (\ref{charlier_difference}) satisfied by Charlier polynomials plays the same role as the inequality (\ref{Chen_Stein_function_diff_ineq}) in the original Stein's approach. These considerations lead us to the inequality of the following theorem which is of independent interest.
\begin{thm}
	\label{inequality_for_independent_case}
	The following inequality holds for \(\chi^2\) distance between generalized binomial distribution and Poisson distribution
	\[
	\chi^2(S_{n},\mathcal{P}(\lambda))
	<\frac{\lambda_2}{\lambda^2}
	\sum_{j=1}^{n}p_j^2\chi^2(S_n-I_j,\mathcal{P}(\lambda))+\frac{1}{2}\left(\frac{\lambda_2}{\lambda}\right)^2.
	\]	
\end{thm}
This inequality provides a bound for the \(\chi^2\) distance of \(S_n\) to Poisson distribution in terms of a  weighted  average of distances of smaller sums \(S_n-I_j\) to Poisson distribution. Iterating this inequality we obtain the following estimate.
\begin{cor}
	\label{cor_upper_bound_for_chi} Let us denote
	\(
	\Theta:=\max_{1\leqslant j\leqslant n} p_j.
	\)
If
	\[
	\Theta^2e^{\Theta}<1
	\]
	then holds inequality
	\[
	\chi^2(S_{n},\mathcal{P}(\lambda))
	\leqslant\frac{1}{2}\left(\frac{\lambda_2}{\lambda}\right)^2
	\Theta^2\frac{2e^\Theta-1}{1- \Theta^2e^\Theta}e^\Theta+
	\frac{\lambda_2\lambda_3}{\lambda^2}e^\Theta
	+\frac{1}{2}\left(\frac{\lambda_2}{\lambda}\right)^2.
	\]
Note that condition \(\Theta^2e^\Theta<1\) is satisfied if \(\Theta\leqslant 0.7\).

\end{cor}
Thus our proof of the above results by Stein-Chen's approach thus is purely probabilistic in nature. The only knowledge assumed on the part of the reader outside the realm of probability is the Charlier-Parseval identity (\ref{Charlier_Parseval_for_chi}) and the two properties satisfied by Charlier polynomials (\ref{recurrence_for_charlier}) and (\ref{charlier_difference}). For the sake of making the paper self contained we provide complete proof these facts in  the appendix of this paper. Our derivation of Charlier-Parseval identity is of independent interest as it has as its starting point a certain integral identity between double integrals from which integral form of Charlier-Parseval identity (\ref{Charlier-Parseval_integral_form}) is established we derive from it the definition of Charlier polynomials and their basic properties most notably their orthogonality with respect to Poisson measure. This approach is analogous to the approach used in  \cite{zacharovas_2019} in the context of Krawtchouk polynomials and Krawtchouk-Parseval identity. A reader interested in a more traditional derivation of these facts may consult, for example, Chapter \(2.81\) of \cite{szego_1975}, Chapter \(7.5\) of  \cite{johnson_2004} or introductory section of our previous paper \cite{zacharovas_hwang_2010}.

Though the upper bounds of this paper for \(\chi^2(S_{n},\mathcal{P}(\lambda))\) in the case of idependent indicators \(S_n\) are not  much stronger than those obtained previously the main focus here is development of a framework of applying Stein-Chen approach to the analysis of \(\chi^2\) distance. The Stein-Chen approach applied to the analysis of total variation distance is notorious for being applicable also to the sums of dependent indicators while treating such sums by analytic methods is very difficult unless the characteristic function such a sums allows for an explicit analytic expression. Thus it is reasonable to expect that the framework we develop will also be applicable to some cases of dependent indicators. We hope to explore this direction in subsequent works.
%
%

%

\subsection{Proofs}

 In what follows we will  denote
\[
S_{n,j}=S_n-I_j.
\]

\begin{proof}[Proof of Theorem \ref{inequality_for_independent_case}] The proof will exploit the expression of \(\chi^2\) by means of Charlier-Parseval identity (\ref{Charlier_Parseval_for_chi}). In order to evaluate its right hand side we will evaluate  mean values \(\mathbb{E}C_k(\lambda,S_n)\) occurring there by means of formula (\ref{recurrence_for_charlier}). The resulting expression will be exactly the Chen-Stein operator  for function \(f(x)=C_{k-1}(\lambda,x)\). Which allows us to proceed in a fairly standard way
\[
\begin{split}
\mathbb{E}C_k(\lambda,S_n)&=\mathbb{E}
\left(\frac{S_n}{\lambda}C_{k-1}(\lambda,S_n-1)
-C_{k-1}(\lambda,S_n)\right)
\\
&=\frac{1}{\lambda}\mathbb{E}
\bigl(S_nC_{k-1}(\lambda,S_n-1)
-\lambda C_{k-1}(\lambda,S_n)\bigr)
\\
&=\frac{1}{\lambda}\sum_{j=1}^{n}\mathbb{E}
\bigl(I_jC_{k-1}(\lambda,S_n-1)
-p_j C_{k-1}(\lambda,S_n)\bigr)
\\
&=\frac{1}{\lambda}\sum_{j=1}^{n}\mathbb{E}
\bigl(I_jC_{k-1}(\lambda,S_{n,j})
-p_j C_{k-1}(\lambda,S_n)\bigr).
\end{split}
\]
Since \(S_n-1=S_{n,j}\) whenever \(I_j=1\).
Hence taking into account independence of \(I_j\) and \(S_{n,j}=S_n-I_j\) we proceed
\[
\begin{split}
\mathbb{E}C_k(\lambda,S_n)&=\frac{1}{\lambda}\sum_{j=1}^{n}p_j \mathbb{E}
\bigl(C_{k-1}(\lambda,S_{n,j})
-C_{k-1}(\lambda,S_n)\bigr)
\\
&=\frac{1}{\lambda}\sum_{j=1}^{n}p_j^2 \mathbb{E}
\bigl(C_{k-1}(\lambda,S_{n,j})
-C_{k-1}(\lambda,S_{n,j}+1)\bigr).
\end{split}
\]
Note that the difference of Charlier polynomials in the above sum can be expressed as a single Charlier polynomial by means of formula (\ref{charlier_difference}) which leads to expression
\[
\mathbb{E}C_k(\lambda,S_n)=-\frac{k-1}{\lambda^2}\sum_{j=1}^{n}p_j^2 \mathbb{E}
C_{k-2}(\lambda,S_{n,j})
\]
for all \(k\geqslant 2\).
Since \(\lambda=\mathbb{E}S_n\) we have \(\mathbb{E}C_1(\lambda,S_n)=\mathbb{E}(S_n/\lambda-1)=0\) and therefore the summation on right-hand side of the Charlier-Parseval identity (\ref{Charlier_Parseval_for_chi}) must start from \(k=2\). 
\[
\begin{split}
\chi^2(S_{n},\mathcal{P}(\lambda))&=\sum_{k=2}^\infty\frac{\lambda^k}{k!}
\bigl|\mathbb{E}C_k(\lambda,S_n)\bigr|^2
\\
&=\sum_{k=2}^\infty\frac{\lambda^k}{k!}
\left|\frac{k-1}{\lambda^2}\sum_{j=1}^{n}p_j^2 \mathbb{E}
C_{k-2}(\lambda,S_{n,j})\right|^2
\\
&=\frac{1}{\lambda^2}\sum_{k=0}^\infty\frac{\lambda^k}{k!}\frac{k+1}{k+2}
\left|\sum_{j=1}^{n}p_j^2 \mathbb{E}
C_{k}(\lambda,S_{n,j})\right|^2
\end{split}
\]
applying here Cauchy inequality we obtain
\[
\begin{split}
	\chi^2(S_{n},\mathcal{P}(\lambda))
	\leqslant\frac{\lambda_2}{\lambda^2}\sum_{k=0}^\infty\frac{\lambda^k}{k!}\frac{k+1}{k+2}
	\sum_{j=1}^{n}p_j^2\bigl| \mathbb{E}
	C_{k}(\lambda,S_{n,j})\bigr|^2.
\end{split}
\]
Noting that \((k+1)/(k+2)<1\), exchanging the order of summation and noting that the inner sums on the right hand are equal to \(\chi^2(S_{n,j},\mathcal{P}(\lambda))\) by (\ref{Charlier_Parseval_for_chi}) we arrive at inequality 
\[
\begin{split}
	\chi^2(S_{n},\mathcal{P}(\lambda))
	&\leqslant\frac{\lambda_2}{\lambda^2}\sum_{j=1}^{n}
	p_j^2\sum_{k=0}^\infty \frac{\lambda^k}{k!}\frac{k+1}{k+2}\bigl|\mathbb{E}
	C_{k}(\lambda,S_{n,j})\bigr|^2
	\\
	&<\frac{\lambda_2}{\lambda^2}\sum_{j=1}^{n}
	p_j^2\left(\frac{1}{2}+\sum_{k=1}^\infty \frac{\lambda^k}{k!}\bigl|\mathbb{E}
	C_{k}(\lambda,S_{n,j})\bigr|^2\right)
	\\
	&=\frac{\lambda_2}{\lambda^2}
	\sum_{j=1}^{n}p_j^2\chi^2(S_{n,j},\mathcal{P}(\lambda))+\frac{1}{2}\left(\frac{\lambda_2}{\lambda}\right)^2.
\end{split}
\]

This completes the proof of the inequality of the theorem.
\end{proof}

We have obtained the above inequality assuming that \(\mathbb{E}S_n=\lambda\). However the sum on the right hand side of the
 inequality  of the last theorem contains \(\chi^2(S_{n,j},\mathcal{P}(\lambda))\) where
\(\mathbb{E}S_{n,j}=\lambda-p_j\). Thus in order to iterate the inequality we must evaluate \(\chi^2(S_{n,j},\mathcal{P}(\lambda))\) in terms of \(\chi^2(S_{n,j},\mathcal{P}(\lambda-p_j))\). Such estimate is provided in the following corollary.
\begin{cor}
	\label{cor_for_chi}
\[
\chi^2(S_{n},\mathcal{P}(\lambda))
<\frac{\lambda_2}{\lambda^2}
\sum_{j=1}^{n}p_j^2e^{p_j}\chi^2(S_{n,j},\mathcal{P}(\lambda-p_j))+
\frac{\lambda_2}{\lambda^2}
\sum_{j=1}^{n}p_j^2(e^{p_j}-1)
+\frac{1}{2}\left(\frac{\lambda_2}{\lambda}\right)^2.
\]
\end{cor}
\begin{proof}
By the definition of \(\chi^2\) distance we can evaluate
\[
\begin{split}
\chi^2(S_{n,j},\mathcal{P}(\lambda))+1&=\sum_{m\ge0}\frac{\bigl(\mathbb{P}(S_{n,j}=m)\bigr)^2}
{e^{-\lambda}\frac{\lambda^m}{m!}}
\\
&= e^{p_j}\sum_{m\ge0}\frac{\bigl(\mathbb{P}(S_{n,j}=m)\bigr)^2}
{e^{-(\lambda-p_j)}\frac{(\lambda-p_j)^m}{m!}}\left(\frac{\lambda-p_j}{\lambda}\right)^m
\\
&\leqslant e^{p_j}\bigl(\chi^2(S_{n,j},\mathcal{P}(\lambda-p_j))+1\bigr)
\end{split}
\]
thus
\[
\chi^2(S_{n,j},\mathcal{P}(\lambda))\leqslant
e^{p_j}\chi^2(S_{n,j},\mathcal{P}(\lambda-p_j))+e^{p_j}-1.
\]
Utilizing the above estimate to evaluate the right hand side of the inequality of Theorem \ref{inequality_for_independent_case} we obtain the estimate of the corollary.
\end{proof}

\begin{lem}
	\label{chi_for_n=1} Suppose \(I\) is an indicator \(p=\mathbb{P}(I=1)=1-\mathbb{P}(I=0)\) then
\[
\chi^2(I,\mathcal{P}(p))= e^p((1-p)^2+p)-1
\]
and as a consequence 
\[
\chi^2(I,\mathcal{P}(p))\leqslant e-1
\]
for all \(0\leqslant p\leqslant 1\).
\end{lem}
\begin{proof}
By the definition  of \(\chi^2\) distance
\[
\chi^2(I,\mathcal{P}(p))+1=\sum_{m=0}^\infty\frac{\bigl(\mathbb{P}(I=m)\bigr)^2}
{e^{-p}\frac{p^m}{m!}}=  \frac{(1-p)^2}
{e^{-p}\frac{p^0}{0!}}+\frac{p^2}
{e^{-p}\frac{p}{1!}}=e^p((1-p)^2+p)
\]
hence follows the first identity of the lemma. The second identity follows from the fact that \( e^p((1-p)^2+p)\) is a monotonously increasing function in the interval \((0,+\infty)\).
\end{proof}

\begin{thm}
	\label{T_upper_bound_for_chi_crude}
Let us denote
\[
\Theta:=\max_{1\leqslant j\leqslant n} p_j 
\]
then
\begin{equation}
	\label{chi_upper_bound_by_max}
\chi^2(S_{n},\mathcal{P}(\lambda))
<\frac{\Theta^2}{2}\frac{2e^\Theta-1}{1- \Theta^2e^\Theta}
\end{equation}
if
\[
\Theta^2 e^{\Theta}<1.
\]
\end{thm}
\begin{proof}

	Note that evaluating \(p_j^2\leqslant p_j\Theta\) we can obtain the inequality
	\[
	\frac{\lambda_2}{\lambda}\leqslant \Theta.
	\]
Applying the above estimate to the inequality of Corollary \ref{cor_for_chi} we obtain
\[
\chi^2(S_{n},\mathcal{P}(\lambda))
< \Theta^2e^\Theta 
\max_{1\leqslant j\leqslant n}\chi^2(S_{n}-I_j,\mathcal{P}(\lambda-p_j))
+\frac{\Theta^2}{2}(2e^\Theta-1).
\]
Iterating the above inequality one more time we get
\[
\begin{split}
\chi^2(S_{n},\mathcal{P}(\lambda))
&< (\Theta^2e^\Theta )^2
\max_{\substack{1\leqslant j,s\leqslant n\\ j\not=s}}\chi^2(S_{n}-I_j-I_s,\mathcal{P}(\lambda-p_j-p_s))
\\
&\quad+
(\Theta^2e^\Theta )\frac{\Theta^2}{2}(2e^\Theta-1)+\frac{\Theta^2}{2}(2e^\Theta-1)
\end{split}
\]
where \(S_{n,j,s}=S_n-I_j-I_s\). Iterating this inequality further on \(n-1\)-th iteration step we get
\[
\chi^2(S_{n},\mathcal{P}(\lambda))
< (\Theta^2e^\Theta )^{n-1}
\max_{1\leqslant j\leqslant n}\chi^2(I_j,\mathcal{P}(p_j))+
\frac{\Theta^2}{2}(2e^\Theta-1)\sum_{j=0}^{n-2}(\Theta^2e^\Theta )^j.
\]
Evaluating here \(\chi^2(I_j,\mathcal{P}(p_j))\) by means of the estimate provided by Lemma \ref{chi_for_n=1} we arrive at inequality
\[
\chi^2(S_{n},\mathcal{P}(\lambda))
< (\Theta^2e^\Theta )^{n-1}
\bigl(e^\Theta((1-\Theta)^2+\Theta)-1\bigr)+
\frac{\Theta^2}{2}(2e^\Theta-1)\sum_{j=0}^{n-2}(\Theta^2e^\Theta )^j
\]
Since
\[
e^\Theta((1-\Theta)^2+\Theta)-1\leqslant \frac{\Theta^2}{2}(2e^\Theta-1)
\]
for all \(0\leqslant\Theta\leqslant1\) therefore
\[\begin{split}
\chi^2(S_{n},\mathcal{P}(\lambda))&<
\frac{\Theta^2}{2}(2e^\Theta-1)\sum_{j=0}^{n-1}(\Theta^2e^\Theta )^j
\leqslant
\frac{\Theta^2}{2}\frac{2e^\Theta-1}{1- \Theta^2e^\Theta}.
\end{split}
\]
The theorem is proved.
\end{proof}
\begin{cor}

	If \(S_n\) is a sum of independent indicators with the same mean equal to \(p\) then
	\[
	\chi^2(S_{n},\mathcal{P}(\lambda))< \frac{p^2}{2}\frac{2e^p-1}{1- p^2e^p}
	\]
	if \(p^2e^p<1\) and \(\lambda=np\).

\end{cor}
\begin{proof}[Proof of Corollary \ref{cor_upper_bound_for_chi}]
	Applying the upper bound (\ref{chi_upper_bound_by_max}) to estimate the \(\chi^2\) distances on the right hand side of the inequality of Corollary \ref{cor_for_chi} we immediately obtain the inequality of the Corollary. 
\end{proof}
\begin{rem}
It is very likely that condition \(\Theta^2 e^{\Theta}<1\) of Corollary \ref{cor_upper_bound_for_chi} could be removed is a more subtle iteration process were used in the proof of Theorem \ref{T_upper_bound_for_chi_crude}.
\end{rem}

\section{Appendix}

\subsection{ Charlier-Parseval Identity and other Properties of Charlier polynomials}

\begin{thm}[Integral form of the Charlier-Parseval identity]
	\label{thm_integral_parseval}
Suppose 
\[
F(z)=\sum_{n=0}^{\infty}a_nz^n
\]
is a generating function of a sequence \(a_0,a_1,\ldots\) then
\[
\sum_{n=0}^{\infty}\frac{|a_n|^2}{e^{-\lambda}\frac{\lambda^n}{n!}}=
\frac{ 1}{2\pi}\int_{0}^{\infty}\left(\int_{-\pi}^{\pi}\Bigl|F\bigl(1+\sqrt{ r/\lambda}e^{it}\bigr)e^{-\lambda\sqrt{r/\lambda}e^{it}}\Bigr|^2e^{- r}\,dt\right)\,dr.
\]
\end{thm}
\begin{proof} Parseval identity applied to function \(F(z)\) takes the form
\[
\sum_{n=0}^{\infty}|a_n|^2r^{2n}=\frac{1}{2\pi}\int_{-\pi}^{\pi}|F(re^{it})|^2\,dt.
\]
Replacing here \(r\to \sqrt{r/\lambda}\) and integrating both sides multiplied by \(e^{-r}\) we obtain the identity
\[
\sum_{n=0}^{\infty}|a_n|^2\frac{n!}{\lambda^n}=\frac{1}{2\pi}\int_{0}^{\infty}\left(\int_{-\pi}^{\pi}|F(\sqrt{r/\lambda}e^{it})|^2\,dt\right)e^{-r}\,dr.
\]
Making a change of variables \(r\to \lambda r^2\) here we obtain the identity
\[
\sum_{n=0}^{\infty}|a_n|^2\frac{n!}{\lambda^n}=\frac{\lambda}{\pi }\int_{0}^{\infty}\left(\int_{-\pi}^{\pi}|F(re^{it})|^2\,dt\right)e^{-\lambda r^2}r\,dr.
\]
The double integral on the right side of the above identity can be regarded as an integral in polar coordinates over all complex plane and as a consequence can be expressed as
\[
\sum_{n=0}^{\infty}|a_n|^2\frac{n!}{\lambda^n}=\frac{\lambda}{\pi}\int_{\mathbb{C}}|F(z)|^2e^{-\lambda|z|^2}\,|dz|
\]
here we denote as \(|dz|=dx\,dy\). Making a change of variables \(z=w+1\) in the above integral we get
\[
\sum_{n=0}^{\infty}|a_n|^2\frac{n!}{\lambda^n}=\frac{1}{\pi}\int_{\mathbb{C}}|F(w+1)|^2e^{-\lambda|w+1|^2}\,|dz|=\frac{1}{\pi}\int_{\mathbb{C}}|F(w+1)|^2e^{-\lambda|w|^2-2\lambda\Re w-\lambda}\,|dz|
\]
which can be expressed as
\[
\sum_{n=0}^{\infty}|a_n|^2\frac{n!}{\lambda^n}=\frac{\lambda e^{-\lambda}}{\pi}\int_{\mathbb{C}}|F(w+1)e^{-\lambda w}|^2e^{-\lambda|w|^2}\,|dz|.
\]
Going back to polar coordinates \(w=re^{it}\), making a change of variables \(r=\sqrt{r/\lambda}\)
and multiplying both sides by \(e^\lambda\) we obtain the identity of the Theorem.
\end{proof}
\begin{cor}
	\label{cor_Parseval}
	Suppose 
	\[
	F(z)=\sum_{n=0}^{\infty}a_nz^n
	\] then
\[
\sum_{n=0}^{\infty}\frac{|a_n|^2}{e^{-\lambda}\frac{\lambda^n}{n!}}=
\sum_{n=0}^{\infty}|\alpha_n|^2\frac{n!}{\lambda^n}
\]
where \(\alpha_0,\alpha_1,\ldots\) are the coefficients in the Taylor expansion of the function
\[
F(z+1)e^{-\lambda z}=\sum_{n=0}^{\infty}\alpha_nz^n.
\]
\end{cor}
\begin{proof} Replacing \(F\bigl(1+\sqrt{ r/\lambda}e^{it}\bigr)e^{-\lambda\sqrt{r/\lambda}e^{it}}\) by 
\[
\sum_{n=0}^{\infty}\alpha_n\bigl(\sqrt{ r/\lambda}e^{it}\bigr)^n
\]
in the main identity of Theorem \ref{thm_integral_parseval} we obtain
\[\begin{split}
\sum_{n=0}^{\infty}\frac{|a_n|^2}{e^{-\lambda}\frac{\lambda^n}{n!}}&=
\frac{ 1}{2\pi}\int_{0}^{\infty}\left(\int_{-\pi}^{\pi}\left|\sum_{n=0}^{\infty}\alpha_n\bigl(\sqrt{ r/\lambda}e^{it}\bigr)^n\right|^2e^{- r}\,dt\right)\,dr
\\
&=
\int_{0}^{\infty}\left(\sum_{n=0}^{\infty}|\alpha_n|^2( r/\lambda)^ne^{- r}\,dt\right)\,dr
\\
&=\sum_{n=0}^{\infty}|\alpha_n|^2\frac{n!}{\lambda^n}.
\end{split}
\]
This completes the proof of the Corollary.
\end{proof}
\begin{cor}
	\label{Cor_orthog_charlier}
	Suppose 
\[
F(z)=\sum_{n=0}^{\infty}a_nz^n\quad \hbox{and}\quad G(z)=\sum_{n=0}^{\infty}a_nz^n
\] then
\[
\sum_{n=0}^{\infty}\frac{a_nb_n}{e^{-\lambda}\frac{\lambda^n}{n!}}=
\sum_{n=0}^{\infty}\alpha_n\beta_n\frac{n!}{\lambda^n}
\]
where \(\alpha_0,\alpha_1,\ldots\) and  \(\beta_0,\beta_1,\ldots\) are the coefficients in the Taylor expansion of the functions
\[
F(z+1)e^{-\lambda z}=\sum_{n=0}^{\infty}\alpha_nz^n\quad\hbox{and}\quad
G(z+1)e^{-\lambda z}=\sum_{n=0}^{\infty}\beta_nz^n.
\]
\end{cor}
\begin{proof} For the proof we write down the identity of Corollary \ref{cor_Parseval} for sequences \(a_n+b_n\) and \(a_n-b_n\) and substracting one resulting identity from the other we immediately obtain the statement of the present Corollary.
\end{proof}
Let us now explore the case of a sequence
\[
a_n=\begin{cases}
	1&\mbox{if } n=k
	\\
	0&\mbox{if } n\not=k
\end{cases}
\] then \(F(z)=z^k\)
and
\[
F(1+z)e^{-\lambda z}=(1+z)^ke^{-\lambda z}=\sum_{m=0}^{\infty}\alpha_m^{(k)}z^m.
\]
The Corollary \ref{Cor_orthog_charlier} ensures that sequences \(\alpha_n^{(x)}\) and \(\alpha_n^{(y)}\) are orthogonal with respect to the measure that is inverse of Poisson measure, that is
\[
\sum_{n\geqslant 0} \alpha_n^{(x)}\alpha_n^{(y)}\frac{n!}{\lambda^n}
=
\begin{cases}
	0 &\mbox{if } x\not=y
	\\
	\frac{1}{e^{-\lambda}\frac{\lambda^{x}}{
			x!}} &\mbox{if } x=y.
\end{cases}
\]
In order to make  sequences \(\alpha_n^{(x)}\) orthogonal with respect to Poisson measure we multiply them by multiplier \(\frac{\lambda^n}{n!}\) and as a result obtain sequences that are polynomials in \(x\) and are the classical Charlier polynomials
\[
C_n(\lambda,x) =\frac{n!}{\lambda^n}\alpha_n^{(x)}
\]
Thus we conclude that polynomials \(C_k(\lambda,x)\) of degree \(k\) in variable \(x\) defined by relation
\begin{equation}
	\label{def_charlier}
\sum_{k\ge0}C_k(\lambda,x) \frac{\lambda^k}{k!}\,w^k
= (1+w)^xe^{-\lambda w}.
\end{equation}
are orthogonal with respect to Poisson measure
\[
\sum_{n\geqslant 0}\frac{\lambda^n}{n!}
C_n(\lambda,x)C_n(\lambda,y)=
\begin{cases}
	0 &\mbox{if } x\not=y
	\\
	\frac{1}{e^{-\lambda}\frac{\lambda^{x}}{
			x!}} &\mbox{if } x=y.
\end{cases}
\]
\begin{thm}[Charlier-Parseval identity] Let \(X\) be an non-negative integer valued random variable then
\[
\sum_{m\ge0}\frac{\mathbb{P}(X=m)^2}
{e^{-\lambda}\frac{\lambda^m}{m!}}
=\sum_{k\ge0}\frac{\lambda^k}{k!}
\bigl|\mathbb{E}C_k(\lambda,X)\bigr|^2.
\]
\end{thm}
\begin{proof}
We will apply the Corollary \ref{cor_Parseval} to the function \(F(z)=\mathbb{E}z^X\). Then
\[a_n=\mathbb{P}(X=n)\]
and 
\[
F(1+z)e^{-\lambda z}=\mathbb{E}e^{-\lambda z}(1+z)^X=\sum_{k\ge0}\mathbb{E}C_k(\lambda,X) \frac{\lambda^k}{k!}\,z^k
\]
by our definition of Charlier polynomials and as a consequence
\[
\alpha_k=\frac{\lambda^k}{k!}\mathbb{E}C_k(\lambda,X). 
\]
Plugging these representations for \(a_n\) and \(\alpha_n\) into the identity of Corollary \ref{cor_Parseval} we immediately obtain the proof of the theorem.
\end{proof}

By our defining relation (\ref{def_charlier}) of Charlier polynomials 
considered  with \(x=m\) \(x=m+1\) and subtracting the resulting identities one from another we obtain
\[
\begin{split}
\sum_{k\ge0}\bigl(C_k(\lambda,m+1)-C_k(\lambda,m)\bigr) \frac{\lambda^k}{k!}\,w^k
&= (1+w)^{m+1}e^{-\lambda w}-(1+w)^{m}e^{-\lambda w}
\\
&=w(1+w)^{m}e^{-\lambda w}
\\
&=\sum_{k\ge0}C_k(\lambda,m) \frac{\lambda^k}{k!}\,w^{k+1}.
\end{split}
\]
comparing now the coefficients at \(w^k\) we obtain the recurrence relation (\ref{charlier_difference}).

In a similar way
\[
\begin{split}
	\sum_{k\ge0}\left(\frac{m}{\lambda}C_k(\lambda,m-1)-C_k(\lambda,m)\right) \frac{\lambda^k}{k!}\,w^k
	&= \frac{1}{\lambda}\left(m(1+w)^{m-1}e^{-\lambda w}-\lambda(1+w)^{m}e^{-\lambda w}\right)
	\\
	&=\frac{1}{\lambda}\frac{d}{dw}(1+w)^{m}e^{-\lambda w}
	\\
	&=\sum_{k\ge0}C_{k+1}(\lambda,m) \frac{\lambda^k}{k!}\,w^{k}.
\end{split}
\]
once again comparing  the coefficients at \(w^k\) we obtain the identity (\ref{recurrence_for_charlier}).

 \section*{Acknowledgments} A part of this paper was written during the authors several visits to Academia Sinica (Taiwan). The author sincerely thanks Prof. Hsien-Kuei Hwang  for his hospitality during the visits.
\bibliographystyle{apalike}

\end{document}